\documentclass[11pt,reqno]{amsart}
\usepackage[T1]{fontenc}
\usepackage{graphicx}
\usepackage{todonotes}
\usepackage{etoolbox}

\usepackage{color}
\definecolor{MyLinkColor}{rgb}{0,0,0.4}
\usepackage{hyperref}

\definecolor{MyLinkColor}{rgb}{0,0,0.4}

\newcommand{\R}{{\mathbb R}}

\newcommand{\Z}{{\mathbb Z}}
\newcommand{\N}{{\mathbb N}}
\newcommand{\s}{\mathbb S}

\newcommand{\wt}{\widetilde}

\newcommand{\p}{\partial}

\newcommand{\e}{\varepsilon}
\newcommand{\om}{\omega}

\newcommand{\J}{ \mathcal J }
\newtheorem{thm}{Theorem}[section]

\newtheorem{lemma}[thm]{Lemma}

\theoremstyle{remark}

\makeatletter
\patchcmd{\maketitle}{\@fnsymbol}{\@alph}{}{}  
\makeatother

\title[Non-uniform continuity of the flow map]{Non-uniform continuity of the flow map for an evolution equation modeling shallow water
waves of moderate amplitude}

\author[N. Duruk Mutluba\c s]{Nilay Duruk Mutluba\c s}
\address{Istanbul Kemerburgaz University, School of Arts and Sciences, Department of Basic Sciences, Mahmutbey Dilmenler Caddesi, No:26, 34217 Ba\u{g}c{\i}lar, Istanbul, Turkey. }
\email{nilay.duruk@kemerburgaz.edu.tr}

\author[A. Geyer]{Anna Geyer}
\address{Departament de Matem\`{a}tiques, Universitat Aut\`{o}noma de Barcelona, Facultat de Ci\`{e}ncies, 08193 Bellaterra, Barcelona, Spain.}
\email{annageyer@mat.uab.cat}

\author[B.--V. Matioc]{Bogdan--Vasile Matioc}
\address{Institut f{\"u}r Angewandte Mathematik, Leibniz Universit{\"a}t Hannover, Welfengarten~1, 30167 Hannover, Germany.}
\email{matioc@ifam.uni-hannover.de}

\date{\today}


\begin{document}

\maketitle

\begin{abstract}
 We prove that the flow map associated to a model equation for surface waves of moderate amplitude 
in shallow water is not uniformly continuous in the Sobolev space $H^s$ with $s>3/2$. 
The main idea is to  consider two suitable sequences of smooth initial data whose difference converges to zero in $H^s$, 
but such that neither of them is convergent.
Our main theorem shows that the   exact solutions corresponding to these sequences of data are uniformly bounded in $H^s$ on a uniform   existence interval, 
but the  difference of the two solution sequences is bounded away from zero in $H^s$  at any positive time in this interval.
The result is obtained by approximating the solutions corresponding to these initial data  by explicit formulae 
and   by  estimating the approximation error in suitable Sobolev norms.
\end{abstract}

\section{Introduction and the main result}

We consider a model equation for surface waves of moderate amplitude in shallow water 
\begin{equation}\label{CH}
 u_t+u_x+6uu_x-6u^2u_x+12 u^3u_x+u_{xxx}-u_{xxt} +14 u u_{xxx}+28u_xu_{xx}=0,
\end{equation}
 which arises as an approximation of the Euler equations in the context of homogenous,
 inviscid gravity water waves. 
 In   recent years, several nonlinear models have been proposed in order to  understand some important aspects of water waves, 
 like wave breaking or solitary waves. 
 One of the most prominent examples is the Camassa-Holm (CH) equation \cite{CH93},
 which is an integrable, infinite-dimensional Hamiltonian system \cite{MKST09, Cons01, CGI06}. 
 The relevance of the CH equation as a model for the propagation of shallow water waves was discussed by Johnson \cite{Joh02},
 where it is shown  that it describes the horizontal component of the velocity field at a certain depth within the fluid;
 see also \cite{Con11}. 
 Building upon the ideas presented in \cite{Joh02}, 
 Constantin and Lannes \cite{CL09}  have recently derived the evolution equation \eqref{CH} as a model for the motion of  the   free 
 surface of the wave, and they  evince  that \eqref{CH}  approximates the governing equations to the same order as 
 the CH equation. 
 Besides deriving \eqref{CH}, the authors of \cite{CL09} also establish  the local well-posedness results for the Cauchy problem associated to \eqref{CH}.
 Relying on a semigroup approach due to Kato \cite{K75}, Duruk \cite{ DM13} has shown that this feature  holds for a  larger class of initial data,
 as well as for solutions which are spatially periodic \cite{DMxx}.
The well-posedness in the context of Besov spaces together with  the regularity and the persistance properties  of  strong  solutions are studied in \cite{MiM13}.
 
Similarly    to  the CH equation, cf.  \cite{CE98, MK98}, the model equation \eqref{CH} can also 
capture the phenomenon of wave breaking: for certain initial data the solution
remains bounded, but its slope becomes unbounded  in finite time cf. \cite{CL09, DMxx}. 
 Unlike for the CH equation, which is known to posses global solutions, cf.   \cite{BC07, CE98},
 it is not 
 apparent how to control the solutions of \eqref{CH} globally, 
 due to the fact that this equation involves higher order nonlinearities in $u$ and its derivatives than the CH equation. 
 On the other hand, if one passes to a moving frame, 
 it can be shown that there exist  solitary travelling wave solutions decaying at infinity \cite{G12}. 
 Their orbital stability has been recently studied in \cite{DG13} using an approach proposed by Grillakis, Shatah and Strauss \cite{GSS87},  which takes
 advantage of the Hamiltonian structure of \eqref{CH}.

In the present paper, we consider the Cauchy problem associated to \eqref{CH} in the setting of periodic functions. 
From the local well-posedness results \cite{DMxx, DM13}, we know that   its solutions depend
continuously on their corresponding initial data   in Sobolev spaces $H^s$ with $s>3/2$. 
Our main result states  that this dependence is not uniformly continuous. 
  This    property was only recently shown to hold true  for the CH equation \cite{HK09, HKM10}, and 
was subsequently  confirmed also for the Euler equations \cite{HM10}  and for several related hyperbolic  problems such as 
the $\mu-b$ equation \cite{LPW13}, the hyperelastic rod equation \cite{Ka10}, for a modified CH system \cite{WL12}, and for the modified CH equation \cite{FL13}.
The main difficulty we encounter compared to all these references is that, as mentioned before, our equation has a higher degree of nonlinearity.
Nevertheless,   we were able two find two sequences of smooth initial data 
whose difference converges to zero in $H^s$, but such that none  of them is convergent, with the corresponding solutions of \eqref{CH} being uniformly 
bounded on a common (nonempty) interval  of existence.
Approximating these solutions by explicit formulae, we then 
successively estimate the error in suitable Sobolev norms and
use  well-known interpolation properties of the Sobolev spaces and commutator estimates to show that 
 at any time of the common existence interval the difference of the two sequences of exact solutions 
 is bounded from below in the $H^s$-norm by a positive constant.    
More precisely,  denoting by $u(\cdot;u_0),$  the unique solution of \eqref{CH}  corresponding  to the initial data $u_0\in H^s(\s)$ with $s>3/2,$ cf. Theorem \ref{T:1},   
our main result states:

\begin{thm}[Non-uniform continuity of the flow map]\label{MT}
 For $s>3/2$, the flow map 
 \[
 u_0\mapsto u(\cdot;u_0): H^s(\s)\to C([0,T), H^s(\s))\cap  C^1([0,T), H^ {s-1}(\s))
 \]
 associated to the evolution equation \eqref{CH} is continuous, but it is  not uniformly continuous.
 More precisely, there exist  two sequences of solutions
  \[(u_n)_{n }, (\wt u_n)_{n }\subset C([0,T_u], H^s(\s))\cap  C^1([0,T_u], H^ {s-1}(\s)),\]
   where $T_u>0$, and a positive constant $C>0$ with the following properties:
  \begin{align*}
   &\sup_{n\in\N}\max_{[0,T_u]}\|u_n(t)\|_{H^s}+\|\wt u_n(t)\|_{H^s}\leq C,\\
   &\lim_{n \to \infty} \|u_n(0)-\wt u_n(0)\|_{H^s}=0,
  \end{align*}
but
\begin{align*}
   &\liminf_{n\to\infty} \|u_n(t)-\wt u_n(t)\|_{H^s}\geq C^{-1}|\sin(t)| \qquad\text{for $t\in(0,T_u]$}.
  \end{align*}
\end{thm}

\medskip
The  structure  of the paper is as follows:  In Theorem \ref{T:1} we recall some properties  concerning  the well-posedness of \eqref{CH} from \cite{DMxx} 
and determine a lower  bound on the existence time of the solution in $H^s$ in terms of the initial data.
Then, we introduce two sequences of   approximate  solutions $(u^{\omega,n})_n,$ $\om\in\{-1,1\},$ and compute    the  approximation error in Lemma \ref{L:Error}.
The corresponding solutions $u_{\om,n}$ of \eqref{CH} determined by the initial data $u^{\om,n}(0)$ are then shown to be uniformly bounded on a common interval of existence,
the absolute error  $\|u^{\om,n}-u_{\om,n}\| $ being computed in different Sobolev norms, cf. Lemmas \ref{L:1}-\ref{L:3}.
We end the paper with the proof of the main result.

\medskip

\paragraph{\bf Notation}

Throughout this paper, 
we shall denote by $C$ positive constants which may depend only upon $s$. 
 Furthermore,  $H^r:=H^r(\s)$, with $r\in\mathbb{R},$ is the $L_2-$based Sobolev space on the circle $\s:=\R/\Z$. 
Given $r\in\R$, we let  $\Lambda^r:=(1-\partial_x^2)^{r/2}$  denote  the Fourier multiplier with symbol $((1+|k|^2)^{r/2})_{k\in\Z}.$
It is well-known that $\Lambda^r:H^q(\s) \to H^{q-r}(\s) $ is an isometric isomorphism      for all $q,r\in\R$.
Furthermore,  the Banach space $W^m_\infty:=W^m_\infty(\s),$ $m\in\N,$ consisting of all
bounded functions which possess   bounded weak derivatives of order less  than  or equal to   $n$, is  endowed 
with the usual norm.
 
 \paragraph{\bf Some useful estimates} The following commutator estimates  play a crucial role in our analysis:
\begin{align}\label{est}
&\|[\Lambda^r, f]g\|_{L_2}\leq C_r\left(\| f_x\|_{L_\infty}\|\Lambda^{r-1}g\|_{L_2}+\|\Lambda^r f \|_{L_2}\|g\|_{L_\infty}\right)\qquad \text{for all $r>3/2$,}\\
\label{CCM}
    & \|[\Lambda^\sigma\p_x,  f] g\|_{L_2}\leq C\|f\|_{H^s}\|g\|_{H^\sigma} \qquad\text{for $s>3/2$ and $1+\sigma\in[0,s]$.}
    &
\end{align}
 They hold for all functions $f,g\in C^\infty(\s)$ and for the commutator $[S,T]:=ST-TS.$ 
 The Calderon-Coifman-Meyer  estimate \eqref{CCM} follows from Proposition 4.2 in Taylor \cite{Tay02}.
  The estimate \eqref{est} is due to Kato and Ponce \cite{KP88, Tay91}.
 Additionally, we shall     use the following multiplier inequality
 \begin{equation}\label{algebra}
 \|fg\|_{H^t}\leq C\|f\|_{H^t}\|g\|_{H^r} \qquad\text{for $t\leq r, \, r> 1/2$}
\end{equation}
 and $f\in H^t(\s)$,\, $g\in H^r(\s)$, cf. e.g. \cite{RS96}. 

 \section{The local well-posedness result}

Using the  above   notation, we observe that  the evolution problem associated to \eqref{CH} can be rendered as the following Cauchy problem:
\begin{equation}\label{PB}
 \left\{
 \begin{array}{rlll}
  u_t&=&u_x+14u u_x+\p_x\Lambda^{-2} R\qquad\text{for $t>0$,}\\
u(0)&=&u_0,
 \end{array}
 \right.
\end{equation}
where $R:=R(u)$ is defined as   
\begin{equation} \label{R}
     R(u):=   7u_x^2 -3 u^4+2u^3-10 u ^2-2u.
\end{equation}

  Relying   upon the local well-posedness results established in \cite{DMxx} for the quasilinear Cauchy problem \eqref{PB},
we determine in the following theorem a lower bound  for the maximal existence time of the solutions in terms of Sobolev norms of  the initial data.
Additionally, we  obtain a  bound on the $H^s$-norm of the local strong solutions on this  particular existence interval.

\begin{thm}\label{T:1}
Let $s>3/2$ be given. Then, we have:
\begin{itemize}
 \item[$(i)$] The problem \eqref{PB} possesses for each $u_0\in H^s(\s)$ a unique maximal solution
 \[u(\cdot;u_0)\in C([0,T), H^s(\s))\cap  C^1([0,T), H^ {s-1}(\s)),\]
whereby $T=T(u_0).$
 Moreover, the flow map
 \[
 u_0\mapsto u(\cdot;u_0): H^s(\s)\to C([0,T), H^s(\s))\cap  C^1([0,T), H^ {s-1}(\s))
 \]
 is continuous. 
 \item [$(ii)$] Given $u_0\in H^s(\s),$ the maximal existence time of the solution $u(\cdot;u_0)$ of
 \eqref{PB} satisfies 
 \begin{equation}\label{Est1}
  T>T_0:=\frac{\| u_0\|_{H^1}^5}{2C \big(1+\| u_0\|_{H^1}^5\big)\|u_0 \|_{H ^s}^{3}}
 \end{equation}
where $C$ is a positive constant.
\item [$(iii)$] We have
 \begin{equation}\label{Est2}
 \|u(t)\|_{H^s}\leq 2\|u_0\|_{H^s}  \qquad \text{for all $t\in[0,T_0].$}
 \end{equation}
\end{itemize}
\end{thm}

Before proceeding with the proof, one can show by using integration by parts shows \cite{CL09, DMxx}  that     the $H^1$-norm of
the solutions of \eqref{CH} is preserved in time when $s\geq 2 $.
Based upon this observation and relying on Theorem \ref{T:1} $(i)$, we then find that  
\begin{align*}
\|u(t)\|_{H^1} =\|u_0\|_{H^1}\qquad \text{for all $u_0\in H^s(\s),\, s>3/2,$ and $0\leq t<T.$}
 \end{align*}

\begin{proof}[Proof of Theorem \ref{T:1}]
The assertion (i) follows from the local well-posedness results established  in \cite{DMxx}.
For $(ii)$, we first pick $ u_0\in H^{s}(\s)$ with $u_0\neq 0$ and denote by $ T$ the maximal existence time of the associated solution $u=u(\cdot;u_0)$. 
In order to determine a lower  bound for  $T$, we first show that  $\| u\|_{H^s}^2$ satisfies a  differential inequality.
We proceed as in  \cite{Mi02, Tay91} and pick a Friedrichs mollifier \footnote
{Choosing $\rho\in C^\infty_0(\R)$ with $\rm{supp}\,  \rho\subset(-1/2,1/2)$ 
 and setting $\rho_\e(x):=\e^{-1}\rho(x/\e)$ for $\e\in(0,1)$ and $x\in\R$, the 
 mollifier $\J_\e$ is defined by  $\J_\e u:=\rho_\e* u$ for all $u\in L_2(\s)$. 
 For every $\e\in(0,1) $ and $s>0$, we have that $\J_\e: L_2(\s)\to H ^k(\s)$ is continuous,
 $\|\J_\e u-u\|_{H^s}\to_{\e\to0}0$ for all $u\in H^s(\s),$ and  $\J_\e:L_2(\s)\to L_2(\s)$ is a contraction. 
 Being a Fourier multiplier, $\J_\e$ commutes with $\p_t$, $\p_x$, and $\Lambda ^s,$ $s>0,$ cf. e.g \cite{BExx}.} $\J_\e\in \mbox{\it OPS}^{-\infty}$, $\e\in(0,1).$ 
 Since $\J_\e$ is itself a Fourier multiplier, the time evolution of the $H^s$-norm of $\J_\e u$ is given by
 \begin{align*}
  \frac{1}{2}\frac{d}{dt}\|\J_\e u\|_{H^s}^2=&\frac{1}{2}\frac{d}{dt}\|\Lambda^s \J_\e u\|_{L_2}^2=\int_{\s} \Lambda ^s\J_\e u\Lambda ^s\J_\e u_t\, dx=I_1+I_2,
 \end{align*}
 where 
\[
   I_1:=\int_{\s} \Lambda ^s \J_\e u\Lambda ^s \J_\e (uu_x)\, dx,
  \qquad I_2:=\int_{\s} \Lambda ^s \J_\e u\Lambda ^s \J_\e (\p_x\Lambda^{-2} R)\, dx.
\]
The latter equality is based on the observation that
\[
 \int_{\s} \Lambda ^s \J_\e u\Lambda ^s \J_\e u_x\, dx=\int_{\s} \Lambda ^s \J_\e u\p_x(\Lambda ^s\J_\e u)\, dx=0.
\]
To estimate the first term, we use the following bound which was derived in Taylor \cite{Tay91}, by means of the  Kato-Ponce estimate \eqref{est}:
\begin{align*}
|I_1|\leq& C\| u\|_{W^1_\infty}\| u\|_{H^s}^2.
\end{align*}
Employing the Cauchy-Schwartz inequality and the algebra property of $H^r(\s)$, $r>1/2,$ the term $I_2$ can be estimated as follows:
\begin{align*}
 |I_2|\leq\|\J_\e \Lambda ^s u\|_{L_2}\|\J_\e\Lambda ^s (\p_x\Lambda^{-2} R)\|_{L_2}\leq C\|u\|_{H^s}\|R\|_{H^{s-1}}\leq C (1+\| u\|_{H^s}^5). 
\end{align*}
 Finally, we combine these estimates and let $\e$  tend to  $0$ to find that
\begin{equation}\label{Est3}
 \frac{d}{dt}\| u\|_{H^s}^2\leq C (1+\| u\|_{H^s}^5)\qquad\text{for all $t\in[0,T)$.}
\end{equation}
Recalling that  the $H^1$-norm of $u$ is preserved in time, we get $\|u_0\|_{H^1}\leq \|u(t)\|_{H^s}$ for all $t\in[0,T),$ and together with 
\eqref{Est3} we find that 
\begin{equation}\label{Est4}
 \frac{d}{dt}\| u\|_{H^s}^2\leq C \frac{1+\| u_0\|_{H^1}^5}{\| u_0\|_{H^1}^5}\| u\|_{H^s}^5\qquad\text{for all $t\in[0,T)$.}
\end{equation}
We conclude that  
\[\displaystyle \|u(t)\|_{H ^s} 
\leq  \frac{\|u_0 \|_{H ^s} }{\Big( 1-C \frac{1+\| u_0\|_{H^1}^5}{\| u_0\|_{H^1}^5}\|u_0 \|_{H ^s}^{3}t\Big)^{1/3}}\quad
\text{for all $t< \max\Big\{ \displaystyle \frac{\| u_0\|_{H^1}^5}{C \big(1+\| u_0\|_{H^1}^5\big)\|u_0 \|_{H ^s}^{3}}, T\Big\}. $}
\]
It follows that    the constant  $T_0$ defined by the relation \eqref{Est1} is a lower bound for $T$, and that $\|u(t)\|_{H ^s}\leq 2\|u_0\|_{H ^s} $ for all $t\leq T_0.$ 
This proves the claim.
\end{proof}


\section{Approximate solutions for the evolution equation} \label{Sect:Approx}

In the following we consider approximate solutions of the evolution equation \eqref{CH} of the form
\begin{equation} \label{approx}
    u^{\omega,n}(t,x):=\frac {\omega n ^{-1}-1-n^{-s}\cos(  n x+\om t  )}{14},
 \end{equation}
where $ \omega\in\{-1,1\}$ and  $n\in \N\setminus\{0\}$. 
When $n$ is very large, the term involving the cosine has a high spatial frequency whereas the other term is constant. 
Before we estimate the error of these approximate solutions, observe that 
\begin{equation}\label{E}
  \|\sin(n x-\alpha)\|_{H^\sigma}=\|\cos (n x-\alpha)\|_{H^\sigma}=\sqrt{\pi} (1+n^2)^{\sigma/2}
\end{equation}
for all $\alpha, \sigma\in\R$ and $n\in \N\setminus\{0\}$. 
Indeed, the functions $\phi_n:=e^{in\cdot}/\sqrt{2\pi}, n\in\Z,$ form an orthonormal basis of $L_2(\s)$, and therefore a 
direct computation  (see also \cite[Lemma 1]{HKM10}) shows that
\begin{align*}
     \|\cos(nx-\alpha)\|_{H^\sigma}^2
     = & \frac{(1+n^2)^\sigma}{2\pi} \Big(\Big|\int_0^{2\pi}\cos(nx-\alpha)e^{-inx}\, dx\Big|^2
         + \Big|\int_0^{2\pi}\cos(nx-\alpha)e^{inx}\, dx\Big|^2\Big)\\
     =& \frac{(1+n^2)^\sigma}{2\pi}\left(\left| \pi e^{-i\alpha}\right|^2
         + \left|\pi e^{i\alpha}\right|^2\right)=\pi(1+n^2)^\sigma.
\end{align*}
We emphasize that in contrast to \cite{HKM10}, due to additional terms appearing in \eqref{CH} the precise computation  of 
\eqref{E} is very important   when estimating the norm of $u^{\om,n}$ in $H^s(\s).$
In view of \eqref{E} and noting that $\|1\|_{H^{\sigma}}=\sqrt{2\pi}$, we obtain the bound
\begin{equation} \label{approxEST}
      \|u^{\omega,n}\|_{H^\sigma} \leq C (1+n^{\sigma-s})\qquad \text{for all $\sigma>0$,  $\omega\in\{-1,1\}$, and  $n\in \N\setminus\{0\}$. }
 \end{equation}
Substituting the approximate solution $u^{\omega,n}$ into the  equation \eqref{CH}   the following expression for the error is found: 
\begin{equation}\label{Error}
 E:=u_t^{\omega,n}-u_x^{\omega,n}-14u^{\omega,n} u_x^{\omega,n}-\p_x\Lambda^{-2} R(u^{\omega,n})
\end{equation}

\begin{lemma}[Estimating the error of approximate solutions]\label{L:Error}
Given $s>3/2$, there is a positive constant $C$ such that 
 \begin{equation}\label{E:1}
  \|E\|_{H^\sigma} \leq C
  \left\{
  \begin{array}{llll}
   {n}^{-2s+1+\sigma}, &\text{if $3/2<s<2,$}\\[1ex]
   {n}^{-s-1+\sigma}, &\text{if $s\geq 2,$}
  \end{array}
  \right.
 \end{equation}
for all $1/2<\sigma\leq1$, $ \omega\in\{-1,1\}$, and $n\in \N\setminus\{0\}$.
\end{lemma}

\begin{proof}
Observe that $E=E_1-E_2$, where
 \begin{align*}
 E_1:=&u_t^{\omega,n}-u_x^{\omega,n}-14u^{\omega,n} u_x^{\omega,n}
         = \frac{n^{-2s+1}}{28}\sin(2(nx+\omega t)),\\
 E_2:=&\Lambda^{-2}\big(14u_x^{\omega,n}u_{xx}^{\omega,n}
     -12(u^{\omega,n})^3u_x^{\omega,n}
     +6(u^{\omega,n})^2u^{\omega,n}_x-20 u^{\omega,n} u_x-2u^{\omega,n}_x\big).
\end{align*}
Recalling \eqref{E} and the fact that $\| u^{\omega,n}\|_{H^{1}}\leq C$ for all $n\geq 1 $, $\omega\in\{-1,1\}$, cf. \eqref{approxEST}, we obtain that 
\begin{align}
  \|E_1\|_{H^\sigma}\leq & C n^{-2s+1+\sigma} \label{E1}\\
 \|E_2\|_{H^\sigma} 
     \leq & C\big( \| u_x^{\om,n}u_{xx}^{\om,n} \|_{H^{\sigma-2}}
         + \| (u^{\om,n})^3u^{\om,n}_x\|_{H^{\sigma-2}}   
         + \|(u^{\om,n})^2u^{\om,n}_x\|_{H^{\sigma-2}}  \nonumber\\
         & \quad + \| u^{\om,n}u^{\om,n}_x\|_{H^{\sigma-2}}  
         + \| u_x^{\om,n} \|_{H^{\sigma-2}}\big) \nonumber\\
     \leq & C\big[ {n}^{-2s+3}\|\sin(2(nx+\omega t))\|_{H^{\sigma-2}}\nonumber\\
     &\quad + \left(\| u^{\om,n} \|_{H^{1}}^3   
         + \|u^{\om,n}\|_{H^{1}}^2
         + \| u^{\om,n}\|_{H^{1}}+1\right) \| u_x^{\om,n} \|_{H^{\sigma-2}} \big] \nonumber\\
     \leq & C( {n}^{-2s+1+\sigma}+{n}^{-s-1+\sigma})\label{E2},
\end{align}
 where we have employed the multiplier inequality \eqref{algebra}.
Combining \eqref{E1} and \eqref{E2} we obtain the desired conclusion.
\end{proof}

\section{Error estimates}

 For each $ \omega\in\{-1,1\}$ and  $n\in \N\setminus\{0\}$, we consider the solution $u_{\om,n}$ of equation \eqref{CH} whose initial
 data is given by the approximate solution  $u^{ \om,n}$ evaluated at time zero, i.e. $u_{\om,n}$ satisfies the equations 
\begin{equation}\label{CP}
 \left\{
     \begin{array}{rlll}
          \p_t u_{\om,n} &=&u_{\om,n} \p_x u_{\om,n}+14u_{\om,n} \p_x u_{\om,n}
            +\p_x\Lambda^{-2} R(u_{\om,n})\qquad\text{$t>0$,}\\
         u_{\om,n}(0) &=&u^{ \om,n}(0).
     \end{array}
 \right.
\end{equation}
 Note that $u_{\om,n}(0)$ is bounded in $H^s(\s)$ for any $s \in \R$. 
 Indeed, since  $u_{\om,n}(0) =u^{ \om,n}(0)$  and recalling the definition \eqref{approx} we find that
\[
    \left|\|u_{\om,n}(0)\|_{H^s}-\frac{\|1\|_{H^s}}{14}\right| 
    \leq\frac{\|\omega n^{-1}\|_{H^s}+n^{-s}\|\cos(nx)\|_{H^s}}{14}
 \]
which  yields
\begin{align*}
\frac{\sqrt{\pi}}{28} \leq \liminf_{n\to\infty}\|u_{\om,n}(0)\|_{H^s} 
        \leq\limsup_{n\to\infty}\|u_{\om,n}(0)\|_{H^s}\leq \frac{\sqrt{2\pi}}{7} 
\end{align*}
for  $\omega\in\{-1,1\}.$
 Furthermore, we obtain  in view of \eqref{E} that 
\begin{align*}
 \lim_{n\to\infty}\|u_{\om,n}(0)\|_{H^1} = \frac{\sqrt{2\pi}}{14}
\end{align*}
for  $\omega\in \{-1,1\}$.
Therefore, if $s>3/2$ we may infer from the Theorem \ref{T:1}  that there exists an integer $n_0\geq1$ and positive constants $C$ and $T_u \leq T_{0}(u_{\om,n}(0))$, such that  
\begin{equation}\label{UE}
 \|u_{\om,n}(t)\|_{H^s}\leq C
\end{equation}
for all $t\in[0,T_u]$, $n\geq n_0$ and $\omega\in \{-1,1\}$.
In the following lemma, we find that the exact
solutions $u_{\om,n}$ have very nice regularity properties,
which allow us to estimate the difference to the approximate solution $u^{\om,n}$ as follows:

\begin{lemma}[Estimating the error $\|u^{\om,n}-u_{\om,n}\|_{H^k}$]\label{L:1} 
Define $k:=s+2$. Then, for each $\omega\in \{-1,1\}$ and  $n\geq n_0$ we have that 
 \[u_{\om, n}\in C([0,T_u], H^{k+1}(\s))\cap C^1([0,T_u], H^k(\s)).\]
 Moreover, there is a constant $C(T_u)>0$ such that
 \begin{equation}\label{DE_2}
\max_{t\in[0,T_u]}\|u^{\om,n}(t)-u_{\om,n}(t)\|_{H^k}\leq C(T_u)n^{2}\qquad\text{for all $\omega\in \{-1,1\}$, $n\geq n_0.$}
 \end{equation}
\end{lemma}

\begin{proof}
Let $\omega\in \{-1,1\}$ and let $n\geq n_0$ be arbitrary. For simplicity we set $u:=u_{\om, n}$.
 Because $u(0)$ is smooth, we know in view of Theorem \ref{T:1} that the Cauchy problem \eqref{CP} has a unique maximal solution  
\[
     u\in C([0,T), H^{k+1}(\s))\cap C^1([0,T), H^k(\s))
 \]
 with existence time  $T:=T(\om,n)$. 
 In order to derive a bound on the absolute error in $H^k$ we have to prove first that this 
 additional regularity holds up to and including the time $T_u$. 
 That is, we have to show that $T >T_u$ for all $\omega\in \{-1,1\}$ and  $n\geq n_0.$ 
 To this end we proceed as in the proof of Theorem \ref{T:1} and compute that
 \begin{align*}
   \frac{d}{dt}\| u\|_{H^k}^2\leq C(\|u\|_{W^1_\infty}\|u\|_{H^k}^2+\|u\|_{H^k}\|\Lambda ^{k-2}\p_x R(u)\|_{L_2})\qquad
   \text{in $[0,T).$}
 \end{align*}
We study the last term more carefully and  obtain in view of  the commutator estimate  \eqref{est} that
 \begin{align*}
      \|\Lambda ^{k-2}\p_x (u_x^2)\|_{L_2}
          \leq &2\|[\Lambda ^{k-2},u_x]u_{xx}\|_{L_2}+2\| u_x\Lambda ^{k-2}u_{xx}\|_{L_2}\\
          \leq &C(\|u\|_{  W^2_\infty }\|u\|_{H^{k-1}}+\|u\|_{  W^1_\infty }\|u\|_{H^{k}}),
 \end{align*}
whereas
 \begin{align*}
     \|\Lambda ^{k-2}\p_x u\|_{L_2}\leq &C\|u\|_{H^{k-1}},
 \end{align*}
 and
 \begin{align*}
   \|\Lambda ^{k-2}\p_x (u^p)\|_{L_2}
       \leq &p\|[\Lambda ^{k-2},u^{p-1}]u_{x}\|_{L_2} 
              +p\| u^{p-1}\Lambda ^{k-2}u_{x}\|_{L_2}\\
        \leq & C(\|u\|_{   W^1_\infty  }^{p-1}\|u\|_{H^{s}}
                   +\|u\|_{   W^1_\infty  }\|u\|_{H^{s}}^{p-1}
                   +\|u\|_{ L_\infty}^{p-1}\|u\|_{H^{k-1}})
 \end{align*}
for $2\leq p\leq 4$. 
In view of the embedding   $ H^{s}(\s)\hookrightarrow{C^1}(\s)$   and \eqref{UE} we find that 
\begin{align*}
   \frac{d}{dt}\| u\|_{H^k}^2\leq C(\|u\|_{H^k}^2+\|u\|_{W^2_\infty}\|u\|_{H^{k-1}} \|u\|_{H^k} )\qquad \text{in $[0,\min\{T, T_u\})$}.
 \end{align*}
 Next we employ the well-known interpolation inequality 
\begin{equation}\label{interpolineq}
      \| u \|_{H^{r}}\leq  \| u \|_{H^{r_1}}^{(r_2-r)/(r_2-r_1)}  \| u \|_{H^{r_2}}^{(r-r_1)/(r_2-r_1)}
\end{equation}
for $r=k-1, r_1=s$ and  $r_2=k$ and obtain that $\|u\|_{H ^{k-1}}^2\leq \|u\|_{H^{s}} \|u\|_{H^k} $ for all $u\in H ^k(\s)$. 
Recalling that   $ H^{k-1}(\s)\hookrightarrow{C^2}(\s),$   we arrive at
\begin{align*}
   \frac{d}{dt}\| u\|_{H^k}^2\leq C  \|u\|_{H^k}^2 \qquad \text{in $[0,\min\{T, T_u\})$}
 \end{align*}
which we may integrate with respect to time to obtain 
\begin{align}\label{ESTa}
   \| u\|_{H^k}\leq e^{C\,T_u} \| u(0)\|_{H^k}\qquad
    \text{in $[0,\min\{T, T_u\})$}.
 \end{align}
This inequality shows that $T>T_u$ for $\omega\in \{-1,1\}$ and for all $n\geq n_0.$ 
Indeed, assuming to the contrary that $T<T_u$, then $￼￼\| u(t)\|_{H^k}\rightarrow \infty$ as $t$ 
approaches the maximal existence time $T$  of $u \in H^k$. 
This is a contradiction to the fact that $u$ is bounded in $H^k$ in view of  \eqref{ESTa}. 
  Finally,  the error estimate \eqref{DE_2} is a simple consequence of \eqref{ESTa} and of the estimate
\[
    \|u(0)\|_{H^k} =\|u^{\omega,n}(0)\|_{H^k} \leq Cn ^{k-s} 
\]
for all $n\geq n_0$, cf. \eqref{approxEST}.
\end{proof}

It turns out that estimate \eqref{DE_2} can be improved when we choose $k=1 $ and $s\geq 2$.
The argument  relies  on the regularity properties derived in the previous Lemma \ref{L:1}.

\begin{lemma}[Estimating the error $\|u^{\om,n}-u_{\om,n}\|_{H^1} $  ]\label{L:2}
Assume that  $s\geq 2.$ 
Then, for all $ n\geq n_0$ and $\omega\in \{-1,1\}$ we have that
 \begin{equation}\label{DE}
\max_{t\in[0,T_u]}\|u^{\om,n}(t)-u_{\om,n}(t)\|_{H^1}\leq C(T_u)n^{-s}.
 \end{equation}
\end{lemma}

\begin{proof}
Denoting the difference between the approximate solution and the exact solution by  $v:=u^{\om,n}-u_{\om,n}$, 
we see that $v$ is a  solution of the initial value problem
\begin{equation}\label{equ}
 \left\{
     \begin{array}{rlll}
           v_t & = v_x-14vv_x+14 u^{\om,n}v_x+14 u^{\om,n}_xv+E+\p_x\Lambda ^{-2}(F)\qquad \text{for $t>0,$}\\
           v(0)&= 0,
     \end{array}
 \right.
\end{equation}
whereby $E$ is the error term defined by \eqref{Error} and
\begin{align*}
  F:=&\,R(u^{\om,n}) - R(u_{\om,n}) \\
     =&\, 14u^{\om,n}_xv_x-7v_x^2-2v-20 u^{\om,n}v+10 v^2+6(u^{\om,n})^2v
           -6u^{\om,n}v^2+2 v^3\\
       & -12 (u^{\om,n})^3v+18 (u^{\om,n})^2v^2-12 u^{\om,n} v^3+3 v^4.
\end{align*}
In view of the regularity property derived for $u_{\om,n}$   in Lemma \ref{L:1}, we may apply $\Lambda^2$ on both sides of  \eqref{equ} and find that
\begin{equation*} 
  \Lambda ^{2}v_t=\Lambda ^{2}v_x-14\Lambda ^{2}(vv_x)+14 \Lambda ^{2}(u^{\om,n}v_x)+14 \Lambda ^{2}(u^{\om,n}_xv)+\Lambda ^{2} E+\p_xF,
 \end{equation*}
 and therewith
 \begin{align*}
  \frac {1}{2}\frac{d}{dt}\|v \|_{H^1}^2=&\int_\s v\Lambda^2 v_t\, dx\\
  =&  \int_\s v\Lambda ^{2}v_x \, dx
  -14 \int_\s v \Lambda ^{2}(vv_x)\, dx
   +14\int_\s v \Lambda ^{2}(u^{\om,n}v_x)\, dx\\
  &+14\int_\s v \Lambda ^{2}(u^{\om,n}_xv)\, dx
  +\int_\s v \Lambda ^{2} E\, dx
  +\int_\s v \p_xF\, dx
 \end{align*}
 for all $t\in[0,T_u].$
Taking into account that 
\begin{align*}
&\int_\s v\Lambda ^{2}v_x \, dx=\int_\s v v_x+v_x v_{xx}\, dx=0
\end{align*}
and noting that
\begin{align*}
&14 \int_\s v \Lambda ^{2}(vv_x)\, dx=14 \int_\s v^2v_x\, dx+7 \int_\s v_x^3\, dx=7 \int_\s v_x^3\, dx
 \end{align*}
 we find 
  \begin{align*}
 \frac{d}{dt}\|v \|_{H^1}^2 
  =&-7 \int_\s v_x^3\, dx+14\int_\s v \Lambda ^{2}(u^{\om,n}v_x)\, dx+14\int_\s v \Lambda ^{2}(u^{\om,n}_xv)\, dx
  +\int_\s v \Lambda ^{2} E\, dx\\
  &-\int_\s v_x \left(14u^{\om,n}_xv_x-7v_x^2-2v-20 u^{\om,n}v+10 v^2+6(u^{\om,n})^2v-6u^{\om,n}v^2\right)\, dx\\
   &-\int_\s v_x \left(2 v^3-12 (u^{\om,n})^3v+18 (u^{\om,n})^2v^2-12 u^{\om,n} v^3+3 v^4\right)\, dx\\
   =&\,14\int_\s v \Lambda ^{2}(u^{\om,n}v_x)\, dx+14\int_\s v \Lambda ^{2}(u^{\om,n}_xv)\, dx
  +\int_\s v \Lambda ^{2} E\, dx\\
  &-\int_\s v_x \left(14u^{\om,n}_xv_x -20 u^{\om,n}v +6(u^{\om,n})^2v-6u^{\om,n}v^2 \right)\, dx\\
   &-\int_\s v_x \left(-12 (u^{\om,n})^3v+18 (u^{\om,n})^2v^2-12 u^{\om,n} v^3 \right)\, dx.
 \end{align*}
 This leads us to the following inequality
 \begin{align*}
 \frac{d}{dt}\|v(t)\|_{H^1}^2 \leq & C\left(\|u^{\om,n}_x\|_{W^1_\infty}\|v\|_{H^1}^2
     +\|E\|_{H^1}\|v\|_{H^1} +
    (1+\|u^{\om,n}\|_{L_\infty})^2\|u^{\om,n}_x\|_{L_\infty}\|v\|_{H^1}^2\right.\\
    &\hspace{0.5cm}\left.
    + \|u^{\om,n}\|_{L_\infty} \|u^{\om,n}_x\|_{L_\infty}\|v\|_{H^1}^3+\|u^{\om,n}_x\|_{L_\infty}\|v\|_{H^1}^4\right).
\end{align*}
Observing that the relation \eqref{approxEST} implies $\sup_{[0,T_u]}\|u^{\omega,n}(t)\|_{H^2}\leq C$,  we find together with \eqref{UE} that
  \begin{align*}
 \frac{d}{dt}\|v \|_{H^1}^2 \leq & C\left(\|u^{\om,n}_x\|_{W^1_\infty}\|v\|_{H^1}^2+\|E\|_{H^1}\|v\|_{H^1} +
\|u^{\om,n}_x\|_{L_\infty}\|v\|_{H^1}^2\right).
 \end{align*}
 Taking now into account the estimates
 \begin{align*}
  \|u^{\om,n}_x\|_{L_\infty}\leq C n^{1-s}\qquad\text{and}\qquad \|u^{\om,n}_x\|_{W^1_\infty}\leq C n^{2-s}
 \end{align*}
for $ n\geq n_0$ and $\omega\in \{-1,1\}$,  we obtain in view of the error estimate \eqref{E:1} in Lemma \ref{L:Error} that 
   \begin{align*}
 \frac{d}{dt}\|v \|_{H^1}^2 \leq & C\left( \|v\|_{H^1}^2+n^{-s}\|v\|_{H^1}\right).
 \end{align*}
The latter estimate leads us to
    \begin{align*}
 \frac{d}{dt}\|v \|_{H^1}  \leq & C\left( \|v\|_{H^1} +n^{-s} \right) \qquad\text{in $[0,T_u]$},
 \end{align*}
  the desired estimate \eqref{DE} following in view of Gronwall's inequality and   since    $v(0)=0$.
\end{proof}

Before proving the main result, we  show the analog of Lemma \ref{L:2} in the situation when $3/2<s<2.$ 
The regularity properties derived in Lemma \ref{L:1} are once again essential.

 \begin{lemma}[Estimating the error $\|u^{\om,n}-u_{\om,n}\|_{H^\sigma} $  ]\label{L:3}
    Let $3/2<s< 2.$ 
    For all $ n\geq n_0$, $\omega\in \{-1,1\}$ and $1/2<\sigma\leq s-1$ we have that   
     \begin{equation}\label{DEmm}
        \max_{t\in[0,T_u]}\|u^{\om,n}(t)-u_{\om,n}(t)\|_{H^\sigma}\leq C(T_u)n^{-s}.
     \end{equation}
 \end{lemma}

\begin{proof}
 In this case, we interpret    the function $v=u^{\om,n}-u_{\om,n}$  as a solution of the initial value problem
\begin{equation}\label{EEE}
 \left\{
     \begin{array}{rlll}
           v_t & =v_x+7((u^{\om,n} +u_{\om,n})v)_x+E+\p_x\Lambda ^{-2} G \qquad\text{for $t>0$},\\
           v(0)&= 0,
     \end{array}
 \right.
\end{equation}
with $E$ given by \eqref{Error}  and with   $G:=\,R(u^{\om,n}) - R(u_{\om,n})$.
It is useful to bring $G$ in the following form 
\begin{align*}
   G =& \,7(u^{\om,n}+u_{\om,n})_xv_x-3v\big( (u^{\om,n})^3+(u^{\om,n})^2u_{\om,n}
            +u^{\om,n}(u_{\om,n})^2+(u_{\om,n})^3)\\
         & +2v\big ((u^{\om,n})^2+u^{\om,n} u_{\om,n}+(u_{\om,n})^2\big)
            -10v\big(u^{\om,n}+ u_{\om,n}\big)-2v.
 \end{align*}
In view of \eqref{EEE}, we have
\begin{align*}
  \frac {1}{2}\frac{d}{dt}\|v\|_{H^\sigma}^2
  =   \int_\s \Lambda^\sigma v\Lambda^\sigma v_t\, dx
  = &\int_\s \Lambda^\sigma v\Lambda^\sigma v_x \, dx
    + \int_\s \Lambda^\sigma v\Lambda^\sigma E\, dx\\
     &+7\int_\s \Lambda^\sigma v\Lambda^\sigma ((u^{\om,n} +u_{\om,n})v)_x \, dx
       +\int_\s \Lambda^\sigma v\Lambda^\sigma \p_x\Lambda ^{-2} G\, dx.
 \end{align*}
The first term in the previous equation vanishes
 \begin{align*}
     \int_\s \Lambda^\sigma v\Lambda^\sigma v_x \, dx=0,
 \end{align*}
while  applying the Cauchy-Schwarz inequality for the second and fourth term we obtain the estimates
 \begin{align*}
    \|\Lambda^\sigma v\Lambda^\sigma E\|_{L_1}
        & \leq \|v\|_{H^\sigma}\|E\|_{H^\sigma}\\
    \|\Lambda^\sigma v\Lambda^\sigma \p_x\Lambda ^{-2} G\|_{L_1}
        & \leq \|v\|_{H^\sigma}\|G\|_{ H^{\sigma-1}}.
 \end{align*}
To derive a bound for the third term,  we use the Calderon-Coifman-Meyer type estimate \eqref{CCM}.
We first commute the operator $\Lambda^{\sigma}\partial_x$ with the function $u^{\om,n} +u_{\om,n}$ and obtain
\begin{align*}
    \int_\s \Lambda^\sigma v\Lambda^\sigma ((u^{\om,n} +u_{\om,n})v)_x\, dx 
    =& \int_\s \Lambda^\sigma v  (u^{\om,n} +u_{\om,n}) \Lambda^\sigma \partial_x v \, dx \\
    & +\int_\s \Lambda^\sigma v [\Lambda^\sigma \partial_x, (u^{\om,n} +u_{\om,n})]v \, dx.\end{align*}
After integrating by parts, we estimate the first integral as follows
\begin{equation*}
   \Big|\int_\s \Lambda^\sigma v  (u^{\om,n} +u_{\om,n}) \Lambda^\sigma \partial_x v \, dx\Big|
      \leq \|\partial_x(u^{\om,n} +u_{\om,n}) \|_{L^\infty} \|v\|_{H^\sigma}^2.
\end{equation*}
To estimate the second integral, we apply the Cauchy-Schwarz inequality and then use the estimate \eqref{CCM} to find
\begin{align*}
 \Big|\int_\s \Lambda^\sigma v [\Lambda^\sigma \partial_x, (u^{\om,n} +u_{\om,n})]v\,dx\Big|
    \leq & \| \Lambda^\sigma v\|_{L^2} \|[\Lambda^\sigma\p_x,u^{\om,n} +u_{\om,n}]v\|_{L_2} \\
    \leq & C\| u^{\om,n} +u_{\om,n}\|_{H^s}\| v\|_{H^\sigma}^2.
\end{align*}
In view of the boundedness of the
family $\{\max_{[0,T_u]}\|u^{\om,n} +u_{\om,n}\|_{H^s}\,:\, n\geq n_0, \, \omega=\pm1\}$ we may combine the preceding estimates and obtain that
 \begin{align*}
     \|\Lambda^\sigma v\Lambda^\sigma ((u^{\om,n} +u_{\om,n})v)_x\|_{L_1}
         \leq C \|v\|_{H^\sigma}^2.
 \end{align*}
The latter argument and the multiplier inequality \eqref{algebra} show that 
\[
\| G\|_{ H^{\sigma-1}}\leq C\|v\|_{H^\sigma},
\]
and together with the error bound \eqref{E:1} obtained in Lemma \ref{L:Error} we conclude that 
\begin{align*}
 \frac{d}{dt}\|v \|_{H^\sigma}^2\leq C\big(\|v\|_{H^\sigma}^2+n^{-2s+1+\sigma}\|v\|_{H^\sigma}\big).
 \end{align*}
 Whence,
 \begin{align*}
 \frac{d}{dt}\|v \|_{H^\sigma} \leq C\big(\|v\|_{H^\sigma} +n^{-2s+1+\sigma} \big) \qquad\text{in $[0,T_u]$},
 \end{align*}
 and the conclusion follows, as in Lemma \ref{L:2}, by taking into account that $-2s+1+\sigma\leq -s$ for all $ \sigma\in(1/2, s-1].$
 \end{proof}

 
 \section{Proof of the main result} 
   In the remaining part we prove that the functions $u_n:=u_{1,n+n_0}$ and $\wt u_n:=u_{-1,n+n_0}$, $n\in\N,$ satisfy all the properties required in Theorem \ref{MT}.  
Recalling the estimate \eqref{UE}, which ensures that the strong solutions $u_{\pm 1,n}$, $n\geq n_0,$ are bounded in $H^s$, proves the first claim
\begin{equation*}
    \sup_{n\geq n_0}\max_{t\in[0,T_u]}\|u_{ 1, n}(t)\|_{H^s}+\| u_{-1,n}(t)\|_{H^s}\leq C,
\end{equation*}
where $T_u$ is the constant introduced right before Lemma \ref{L:1}.
The second assertion follows   by taking into account the definition of the approximate  solutions \eqref{approx}, which yields
\[
     \|u_{ 1,n}(0)-u_{- 1,n}(0)\|_{H^s}=\frac{2n^{-1}}{14}\|1\|_{H^s}\to_{n\to\infty}0.
 \]
To show that the third claim of Theorem \ref{MT} holds, we have to derive a decay estimate for the difference between the two unknown exact solutions. 
The trick is to work with inequalities involving the estimates for the absolute errors   deduced in the preceding lemmas. 
We assume first that $s\geq 2$, and observe that 
\begin{align}
  \|u_{1,n}(t)-u_{- 1,n}(t)\|_{H^s} \geq& \|u^{1,n}(t)-u^{-1,n}(t)\|_{H^s} \notag\\ 
                                                                & -\|u^{1,n}(t)-u_{1,n}(t)\|_{H^s} -\|u^{-1,n}(t)-u_{- 1,n}(t)\|_{H^s}\label{DE-1}
\end{align}
for all $t\in[0,T_u] $ and $n\geq n_0.$ 
Now we find lower bounds for each of these three terms. 
A simple calculation yields that
\begin{align}\label{DE0}
     \|u^{1,n}(t)-u^{-1,n}(t)\|_{H^s} 
           = & \frac{1}{14}\|2n^{-1}-n^{-s}\left(\cos(nx+t)-\cos(nx-t)\right)\|_{H^s}\nonumber\\
      \geq & \frac{ n^{-s}}{7}|\sin(t)|\,\|\sin(nx)\|_{H^s}-\frac{ \sqrt{2\pi}n^{-1}}{7}\nonumber\\
      \geq & \frac{\sqrt{\pi}|\sin(t)| }{7}  -\frac{ \sqrt{2\pi}n^{-1}}{7}
\end{align}
for all $n\geq n_0$ and $t\in[0,T_u]$, where we have used \eqref{E} in the last inequality.
To estimate the second and third term in \eqref{DE-1}, we apply the interpolation inequality \eqref{interpolineq} with $r=s$, $r_1=1$ and $r_2=k$, and find  
\begin{align}\label{DE4}
    \|u^{\pm 1,n}(t)-u_{\pm 1,n}(t)\|_{H^s}
        \leq & \|u^{\pm 1,n}(t)-u_{\pm 1,n}(t)\|_{H^1}^{\frac{2}{k-1}}
                   \|u^{\pm 1,n}(t)-u_{\pm 1,n}(t)\|_{H^k}^{\frac{s-1}{k-1}} \nonumber \\
        \leq  & C(T_u) \,n ^{\frac{-2}{k-1}}
\end{align}
for all $n\geq n_0$, in view of the estimates \eqref{DE} and \eqref{DEmm} obtained in Lemma \ref{L:1} and \ref{L:2}, respectively. 
Gathering \eqref{DE0} and \eqref{DE4}, we obtain that
\begin{equation*}
    \|u_{1,n}(t)-u_{-1,n}(t)\|_{H^s} 
        \geq  \frac{\sqrt{\pi}|\sin(t)| }{7}  -\frac{ \sqrt{2\pi}n^{-1}}{7}-C(T_u) n ^{\frac{-2}{k-1}},
\end{equation*}
for all $t\in[0,T_u]$ and  $n\geq n_0$. 
Finally, we let $n\rightarrow \infty$ to  complete the proof in the case $s\geq 2$. 

For $3/2<s<2$ we note that the estimate \eqref{DE0} is still valid, whereas the analog of \eqref{DE4} holds in view of  Lemma \ref{L:3}. Indeed, we find that 
 \begin{align*}
      \|u^{\pm 1,n}(t)-u_{\pm 1,n}(t)\|_{H^s}
             \leq & \|u^{\pm 1,n}(t)-u_{\pm 1,n}(t)\|_{H^\sigma}^{\frac{2}{k-\sigma}}
                        \|u^{\pm 1,n}(t)-u_{\pm 1,n}(t)\|_{H^k}^{\frac{s-\sigma}{k-\sigma}} \\
             \leq & C(T_u)\, n ^{\frac{-2\sigma}{k-\sigma}},
\end{align*}
 where $\sigma\in(1/2,s-1]$ is fixed and $n\geq n_0$.
The final argument of the proof is analogous to the one presented in the case when $s\geq2$.
\qed

\subsection*{Acknowledgements}
A. Geyer was supported by the FWF project J 3452 ''Dynamical Systems Methods in Hydrodynamics`` of the Austrian Science Fund.

\end{document}